\documentclass[12pt,a4paper]{amsart}
\usepackage{url}
\usepackage{amsmath,amsthm,amsfonts,amssymb,latexsym}

\newtheorem{theorem}{Theorem}[section]
\newtheorem{proposition}[theorem]{Proposition}
\newtheorem{lemma}[theorem]{Lemma}

\newtheorem{observation}[theorem]{Observation}

\theoremstyle{definition}
\newtheorem{definition}[theorem]{Definition}

\newtheorem{question}[theorem]{Question}

\newcommand{\U}{\mathcal U}
\newcommand{\w}{\omega}

\newcommand{\IP}{\mathbb P}

\newcommand{\B}{\mathcal{B}}
\newcommand{\C}{\mathcal{C}}

\newcommand{\A}{\mathcal{A}}

\newcommand{\F}{\mathcal{F}}

\newcommand{\V}{\mathcal{V}}

\newcommand{\bigvid}{\hat{\ \ }}

\newcommand{\uhr}{\upharpoonright}

\newcommand{\name}[1]{\dot{#1}}
\newcommand{\la}{\langle}
\newcommand{\ra}{\rangle}

\newcommand{\forces}{\Vdash}

\newcommand{\hot}{\mathfrak}

\newcommand{\nothing}[1]{}

\newcommand{\zrost}{\w^{\uparrow\w}}

\title[Selection Principles in the Laver and other models]{Selection Principles
in the Laver, Miller, and Sacks models}

\author{Lyubomyr Zdomskyy}

\address{Kurt G\"odel Research Center for Mathematical Logic,
University of Vienna, W\"ahringer Stra\ss e 25, A-1090 Wien,
Austria.}
\email{lzdomsky@gmail.com}
\urladdr{http://www.logic.univie.ac.at/\~{}lzdomsky/}

\subjclass[2010]{Primary: 03E35, 54D20. Secondary: 54C50, 03E05.}
\keywords{Menger space, Hurewicz space,
 Laver forcing, Miller forcing, Sacks forcing, fusion.}
\thanks{The author would
like to thank  the Austrian Science Fund FWF (Grants I 2374-N35   I 3709-N35)
 for generous support of this research.}

\begin{document}
\maketitle

\begin{abstract}
This article is devoted to the interplay between forcing with fusion
and  combinatorial covering properties. We discuss known instances of this interplay as well as present  a new one,
namely
that in the Laver model
 for the consistency of the Borel's conjecture,
the Hurewicz property is preserved by finite products of metrizable spaces.
\end{abstract}

\section{Introduction}

Combinatorial covering properties (or selection principles) arguably arose from the study
of special sets of reals. These  resolved many classical
questions in general topology and measure theory. As a result,
information about special sets of reals is included in standard
topology textbooks, such as Kuratowski's Topology \cite{Kur92}. The most
influential survey on special sets of reals is, probably, Miller's
chapter \cite{Mil84} in the Handbook of Set-Theoretic Topology. The
most recent monograph on this topic is written by Bukovsky, see
\cite{Buk11}. It complements nicely the classical book \cite{BarJud95} of
Bartoszynski and Judah. This theory still finds interesting applications
in general topology, see, e.g., \cite{HruMil??} for the interplay between
$\lambda$-sets and homogeneity.

A typical example of the evolution of special sets of reals into selection principles comes from
  \emph{strong measure zero} sets (SMZ sets in what follows):
$X\subset\mathbb R$ is SMZ if for every sequence
$\la\epsilon_n:n\in\w\ra$ of positive reals there exists a sequence
$\la a_n:n\in\w\ra$ of ``centers'' such that
$X\subset\bigcup_{n\in\w}B(a_n,\epsilon_n)$, where
$B(a,\epsilon)=\{x\in\mathbb R:|x-a|<\epsilon\}$. SMZ sets were
introduced by Borel ca. a century ago  who conjectured  \cite{Bor19}
that only countable sets have such property, i.e., that there are
only trivial examples of SMZ sets. This conjecture has been
consistently refuted by Sierpi\'{n}ski \cite{Sie28} in 1928, but the
question whether Borel's conjecture (BC in what follows) is
consistent was answered only in 1976 by Laver in his seminal paper
\cite{Lav76}. This outstanding result was the
first\footnote{According to our colleagues who worked in set theory
already in the 70s.} instance when a forcing, adding a real, was
iterated with countable supports without collapsing cardinals. This
work of Laver can be thought of as  a ``birth'' of fusion, the
latter being a kind of  a gentle weakening of the countable
completeness which allows  to add  new reals and nonetheless
``treat'' countably many dense sets with a single condition, and
thus it is also one of the motivations behind   Baumgartner's axiom
$A$  and  Shelah's theory  of proper forcing.

In 1938 Rothberger  worked on the question whether having  SMZ is  a topological property, see \cite{Ro38},
and introduced the following topological counterpart  of SMZ sets which is now known as Rothberger's property and is one of the most intensively studied
 covering properties:
  For every sequence $\la \U_n : n\in\omega\ra$
of open covers of $X$ there exists a sequence $\la U_n : n\in\omega \ra$ such that
 $U_n\in\U_n$ for all $n$ and the collection $\{U_n:n\in\omega\}$
is a cover of $X$. Three years later he has shown in \cite{Ro41} that $\hot b=\w_1$ implies the existence of a
SMZ set which is not Rothberger and hence that having SMZ is not a topological property.
Recall that $\hot b$ is the minimal cardinality of a
subset of $\w^\w$ which cannot be covered by a $\sigma$-compact one.
In the same paper he shows that the BC is equivalent to all
Rothberger metrizable spaces being countable. The Rothberger property is an example of strong
combinatorial covering properties  in the sense that one selects a \emph{single} element of each cover,
aiming at a ``diagonalizing'' cover of certain kind. These proved to be important
in the study of local properties of spaces of continuous functions, see \cite{Ark92} and references therein.

Another source of selection principles was the dimension theory, where the following property takes
its origin:
A topological space
$X$ has the  \emph{Menger} property (or, alternatively, is a Menger space)
 if for every sequence $\la \U_n : n\in\omega\ra$
of open covers of $X$ there exists a sequence $\la \V_n : n\in\omega \ra$ such that
each $\V_n$ is a finite subfamily of $\U_n$ and the collection $\{\cup \V_n:n\in\omega\}$
is a cover of $X$. This property was introduced by  Hurewicz, and the current name
(the Menger property) is used because Hurewicz
proved in   \cite{Hur25} that for metrizable spaces his property is equivalent to
one  property of a base considered by Menger in \cite{Men24}.
If in the definition above we additionally require that $\{\cup\V_n:n\in\w\}$
is a \emph{$\gamma$-cover} of $X$
(this means that the set $\{n\in\w:x\not\in\cup\V_n\}$ is finite for each $x\in X$),
then we obtain the definition of the \emph{Hurewicz  property}  introduced
in \cite{Hur27}. Each $\sigma$-compact space is obviously a Hurewicz space, and Hurewicz spaces have the
Menger property. Contrary to a conjecture of Hurewicz
the class of  metrizable spaces having the Hurewicz property
 appeared  to be much wider than the class of $\sigma$-compact spaces \cite[Theorem~5.1]{COC2}.
Moreover, there always exists a subspace of $\w^\w$ of cardinality $\hot b$ all of whose finite powers are
Hurewicz,  see \cite{BarShe01, BarTsa06}.
Thus we can say that there are always \emph{non-trivial} examples  of metrizable Hurewicz spaces, as contrasted
with the behaviour of Rothberger spaces in the Laver model.
The properties of Menger and Hurewicz are classical examples of \emph{weaker}
combinatorial covering properties, i.e., properties obtained by selecting finitely many
elements from each cover in a given sequence.

The Rothberger property and SMZ sets in the measure theory as well
as properties of Menger and Hurewicz in the dimension theory,
respectively, are not among the most intensively investigated ones
in these areas of mathematics. The reason is that   their study
required methods developed decades after their introduction. Namely,
one of the most fascinating features of these properties as well as
many other selection principles is their behavior in various models
of set theory, which makes forcing and cardinal characteristics one
of the main tools in the area. This became apparent from at least
the publication of   \cite{Lav76}, in  \cite{BarDow12, ChoHru???,
ChoRepZdo15, MilTsa10, SchTal10} one can find more recent works
along these lines. In what follows we discuss the affect of three
classical forcing notions, namely Laver, Miller, and Sacks ones, on
selection principles. The unifying theme of all of them is
\emph{fusion} based on the fact that the conditions are certain
trees on the natural numbers.
\medskip


\noindent The next after \cite{Lav76} application of the
\textbf{\emph{Laver forcing}} to combinatorial covering properties
seems to be \cite{Dow90} where it is proved that in the Laver
model, the properties $\alpha_1$ and $\alpha_2$ introduced by
Arkhangel'ski\u{\i} in \cite{Ark79} coincide. Recall that a space $X$
satisfies $\alpha_2$ (resp. $\alpha_1$) if for every countable
family $\mathcal S$ of non-trivial sequences converging to a point
$x\in X$ there exists a sequence $T$ converging to $x$ such that
$|T\cap S|=\w$ (resp. $S\subset^* T$) for all $S\in\mathcal S$.
How is this related to covering properties? It has been proved in  \cite{TsaZdo12} that  for a metrizable space
$X$, property $\alpha_1$ of $C_p(X)$, the space of all continuous
functions $f:X\to\mathbb R$ with the topology inherited from the
Tychonoff product $\mathbb R^X$, is equivalent to $X$ having the
Hurewicz property with respect to all countable Borel covers.  Due
to the fact that (unlike the open sets) Borel sets are closed under
countable intersections, this property is equivalent to
$S_1(\B_\Gamma,\B_\Gamma)$ asserting that for every sequence
$\la\U_n:n\in\w\ra$ of Borel $\gamma$-covers of $X$ there exists
$\la U_n\in\U_n:n\in\w\ra$ which is a $\gamma$-cover of $X$.
Replacing here ``Borel'' with ``clopen'' one gets the property
$S_1(\C_\Gamma,\C_\Gamma)$ equivalent to $\alpha_2$ of $C_p(X)$, see
\cite{Sak07}. One of the most intriguing conjectures along these lines
is due to Scheepers \cite{Sch99} and says that for subsets of
$\mathbb R$ the properties $S_1(\C_\Gamma,\C_\Gamma)$ and
$S_1(\Gamma,\Gamma)$ coincide, the definition of the latter obtained
from that of $S_1(\C_\Gamma,\C_\Gamma)$ be replacing ``clopen'' with
``open''. It follows from the above that for metrizable spaces we
have $S_1(\B_\Gamma,\B_\Gamma)=S_1(\C_\Gamma,\C_\Gamma)$ in the
Laver model, and by
$$ S_1(\B_\Gamma,\B_\Gamma)\to S_1(\Gamma,\Gamma)\to S_1(\C_\Gamma,\C_\Gamma)$$
 the  Scheepers' Conjecture is
true in the Laver model. However, it is totally unclear what
happens under CH. Moreover, if there is $X\subset\mathbb R$
satisfying $S_1(\C_\Gamma,\C_\Gamma)$ but not $S_1(\Gamma,\Gamma)$,
then this is also obviously true in any forcing extension with the
same reals, and hence we can get such a set in a model of CH. Thus
if the negative answer to Scheepers' question is consistent, it is
also consistent with CH.

Using the key  lemma of \cite{Lav76} allowing to analyze names for
reals in the Laver model, Miller and Tsaban
proved \cite{MilTsa10} that all $X\subset\mathbb R$ satisfying
$S_1(\B_\Gamma,\B_\Gamma)$ must have size $\leq \w_1<\hot b=\hot c$
in this model, and hence this is also true for $S_1(\C_\Gamma,\C_\Gamma)$.
This answered the question whether in ZFC there  exists a set of
reals satisfying $S_1(\C_\Gamma,\C_\Gamma)$ of size $\hot b$ in the
negative. Their application of Laver's analysis of names of reals
found further application to covering properties of products.

One of the basic questions about a topological property is whether
it is preserved by finite products.
 In case of combinatorial covering properties we know that  the strongest
 possible  negative result is consistent: Under CH there exist
 $X,Y\subset\mathbb R$ which have the $\gamma$-space property with
 respect to countable Borel covers\footnote{This property is stronger than $S_1(\B_\Gamma,\B_\Gamma)$
 and we refer the reader to \cite{MilTsaZso16} for its definition.},
 whose product $X\times Y$ is not
 Menger, see \cite[Theorem~3.2]{MilTsaZso16}. Thus the product of
 spaces with the strongest combinatorial covering  property considered thus far
 might fail to have even the weakest one. This implies that no
 positive results for combinatorial covering properties can be obtained outright
 in ZFC.
  Unlike the vast majority of
 topological and combinatorial consequences under  CH, the latter
 one does not follow from any equality among cardinal
 characteristics of the continuum, see \cite{BarDow12} and the discussion on \cite[p. 2882]{MilTsaZso16}.
However, there are many other negative results stating that under
certain equality among cardinal
 characteristics (e.g., $\mathit{cov}(\mathcal N)=\mathit{cof}(\mathcal N)$, $\hot b=\hot
d$, etc.\footnote{We refer the reader to \cite{Bla10} for the
definitions and basic properties of cardinal characteristics of the
continuum which are mentioned but are not used in the proofs in this
article.}) there are spaces $X,Y\subset \mathbb R$ with some combinatorial covering
property such that $X\times Y$ is not Menger, see, e.g.,
\cite{Bab09, RepZdo10, TsaSch02}.

Regarding the positive results,
 until recently the most
unclear situation was with the  Hurewicz  property and the weaker
ones.
 There are two reasons why a product of Hurewicz spaces
$X,Y$ can fail to be  Hurewicz. In the first place, $X\times
Y$ may simply fail to be a Lindel\"of space, i.e., it might have an
open cover $\U$ without countable subcover.  This may indeed happen:
in ZFC there are two normal spaces $X,Y$ with a covering property
much stronger than the Hurewicz one such that $X\times Y$ does not
have the  Lindel\"of property, see \cite[Section 3]{Tod95}. However,
the above situation becomes impossible if we restrict our attention
to metrizable spaces. This second case
 turned out to be  sensitive to
the ambient set-theoretic universe: under CH there exists a Hurewicz space
whose square is not Menger, see \cite[Theorem~2.12]{COC2}. The above result has been achieved by
a transfinite construction of length $\w_1$, using the combinatorics of the ideal of
 measure zero  subsets of reals. This combinatorics turned out
 \cite[Theorem~43]{TsaSch02} to require
 much weaker set-theoretic assumptions  than CH.
 In particular, under the Martin
Axiom there are Hurewicz subspaces of the irrationals whose product is not Menger.

The following theorem, which is the main result which we are going
to prove in this article, shows that an additional assumption in the
results from \cite{COC2,TsaSch02} mentioned above is really needed.

\begin{theorem} \label{main}
In the Laver model for the consistency of the Borel's conjecture,
the  Hurewicz  property is preserved by finite products of
metrizable spaces. Consequently, the product of any two Hurewicz
 spaces is Hurewicz if it is Lindel\"of.
\end{theorem}

Let us recall that for metrizable spaces, the separability is
equivalent to being Lindel\"of, see, e.g.,
\cite[Theorem~4.1.15]{Eng}.

This theorem is a further improvement of the main result of
\cite{RepZdo17} which states that product of two Hurewicz subspaces
of $\mathbb R$ is Menger in the Laver model. The conclusion of
Theorem~\ref{main} does not follow from
 Borel's conjecture:
If we  add  $\w_2$ many random reals  over the Laver model then
Borel's conjecture still holds by \cite[Section 8.3.B]{BarJud95}  and  we have
$\mathit{cov}(\mathcal N)=\mathit{cof}(\mathcal N)$,
and hence in this model
there exists  a Hurewicz set of reals  whose square is not Menger, see \cite{TsaSch02}.
Thus  Borel's conjecture is consistent with the  existence of  a Hurewicz set of reals with even non-Menger square.

Theorem~\ref{main} seems to be an instance of a more general
phenomena, namely that proper posets with fusion affect the behavior
of combinatorial covering properties. This happens because sets of
reals with certain combinatorial covering properties are forced  to
have a rather clear structure, which   suffices  to prove  positive
preservation results. For instance, the core of the proof of
Theorem~\ref{main} is that Hurewicz subspaces  of the real line are
concentrated in a sense around their ``simpler'' subspaces.

Later we shall discuss  similar behavior of combinatorial covering
properties in the Miller and Sacks   models. Now, however, we would
like to present a ``side effect'' of the proof of Theorem~\ref{main}
for maximal almost disjoint families of subsets of $\w$.
%
Recall that an infinite $\A\subset[\w]^\w$ is called a \emph{mad
family}, if $|A_0\cap A_1|<\w$ for any distinct $A_0,A_1\in \A$, and
for every $B\in[\w]$ there exists $A\in\A$ such that
$|B\cap A|=\w$. In \cite[Theorem~2.1]{Bre98} Brendle constructed
under CH a mad family $\A$ on $\w$ such that the Mathias
forcing\footnote{Since we shall not analyze this poset directly but
rather use certain topological characterizations, we refer the
reader to, e.g., \cite{Bre98} for its definition.} $M_{\F(\A)}$
associated to the filter
$$ \F(\A)=\{F\subset\w:\exists \A'\in[\A]^{<\w}(\w\setminus\cup\A'\subset^* F)\}$$
adds a dominating real. In the same paper Brendle asked whether such
a mad family can be constructed outright in ZFC. This question has
been answered in the affirmative in \cite{GuzHruMar??} and later
independently also in \cite{ChoRepZdo15} using different methods.
However, the mad families constructed there had topological copies
of the Cantor set inside, and hence by a  simple absoluteness
argument are destroyed (i.e., are not maximal any more) by any
forcing adding new reals. One can refine Brendle's question in the
following way: Suppose that a mad family $\A$ cannot be destroyed by
some very ``mild''  forcing $\IP$, i.e., it remains maximal in
$V^{\IP}$, must then $M_{\F(\A)}$ add dominating reals? This
refinement makes sense  in cases when one is interested in
 destroying mad families  without adding dominating
reals and takes some $\IP$ which does not  add them. Indeed, if $\A$
is already destroyed by $\IP$, there is no need to use its Mathias
forcing for its destruction in a hypothetic construction of a model
where, e.g., $\hot b$ should stay small.  In what follows we
concentrate on the case  $\IP$ being the Cohen forcing $\mathbb C$.
Mad families $\A$ which remain maximal in $V^{\mathbb C}$ will be
called Cohen-indestructible. The following theorem has been proved
in \cite{AurZdo??}

\begin{theorem} \label{p_equals_c}
$\hot p=\mathit{cov}(\mathcal{N})=\hot c$ implies the existence of a
Cohen-indestructible mad family $\A$ such that $M_{\F(\A)}$ adds a
dominating real.
\end{theorem}

Recall from \cite{Kur01} that a mad family $\A$ is called
\emph{$\w$-mad} if for every sequence $\la X_n:n\in\w\ra$ of
elements of $\F(\A)^+$ there exists $A\in\A$ such that $|A\cap
X_n|=\w$ for all $n$. Cohen-indestructible mad families are closely
related to $\w$-mad ones, see \cite{Mal90} or
\cite[Theorem~4]{Kur01}: Every $\w$-mad family is
Cohen-indestructible, and if $\A$ is Cohen-indestructible, then for
every $X\in\F(\A)^+$ there exists $Y\subset X, $ $Y\in\F(\A)^+$,
such that $\A\uhr Y=\{A\cap Y:A\in\A, A\cap Y$ is infinite$\}$ is
$\w$-mad as a mad family on $Y$.

The proof of Theorem~\ref{p_equals_c}  actually gives an $\w$-mad
family. The next theorem shows that $\hot b=\hot c$ would not
suffice in Theorem~\ref{p_equals_c} in a strong sense.

\begin{theorem} \label{main_mad}
In the Laver model for the consistency of the Borel conjecture, for
every $\w$-mad family $\A$ the poset $M_{\F(\A)}$ preserves all
ground model unbounded sets. In particular, if $\A$ is
Cohen-indestructible, then there exists $X\in\F(\A)^+$ such that
$M_{\F(\A)\uhr X}$ preserves all ground model unbounded sets, where
$\F(\A)\uhr X$ denotes the filter on $\w$ generated by the centered
family $\{F\cap X:F\in\F(\A)\}$.
\end{theorem}

Theorem~\ref{main_mad} improves our earlier result proved in
\cite{AurZdo??} asserting, under the same premises, that
$M_{\F(\A)\uhr X}$ keeps the ground model unbounded.
  In our
proof of Theorem~\ref{main_mad}  we shall not work with the Mathias
forcing directly, but rather use the following characterization
obtained in \cite{ChoRepZdo15}: For a filter $\F$ on $\w$ the poset
$M_{\F}$ adds no dominating reals (resp. preserves all ground model
unbounded sets) iff $\F$ has the Menger (resp. Hurewicz) covering
property when considered with the topology inherited from $\mathcal
P(\w)$, which is identified with the Cantor space $2^\w$ via
characteristic functions. Thus Theorem~\ref{main_mad} says that if
$A$ is an $\w$-mad family, then $\F(\A)$ is Hurewicz, whereas its
predecessor from \cite{AurZdo??} gave only the Menger property.
Let us note  that there are ZFC examples of Menger
non-Hurewicz filters on $\w$, see \cite{RepZdoZha14}.
 Theorem~\ref{main_mad} also intuitively
explains why the open question whether there exists an $\w$-mad
family outright in ZFC (it is known and rather easy that they exist
under $\hot b=\hot c$, e.g., in the Laver model) is so difficult: Any
such ZFC example should ``often''  give a mad family whose dual filter is Hurewicz,
and it is even unknown how to get one with the Menger filter
in ZFC, see \cite{GuzHruMar??}. This motivates the following
\begin{question}
Is there a ZFC example of a mad family $A$ such that $\F(\A)$ is
Hurewicz?
\end{question}
If the answer to this problem is affirmative, then it should
probably require some new ideas since with the existing methods it
is unclear even how to get a Menger one. On the other hand, a negative answer might
be surprisingly easy.
\medskip


\noindent Superficially the conditions in the \textbf{\emph{Miller
forcing}} introduced in \cite{Mil83} are very similar to those in the Laver one, the only
difference being that the splitting does not have to occur so often.
However, these posets have completely different affects on the
behavior of selection principles discussed above. In this context the Miller model has
been first investigated in \cite{Zdo05}, where it has been proved
that for sets of reals, the Rothberger property implies the Hurewicz
one in this model. This is of course  also true in the Laver
model because the Borel's conjecture holds there. On the contrary,
in the Miller model there are uncountable Rothberger spaces
(see \cite{OreTsa11} for a much more stronger
result and references therein), so here this implication holds for
different, qualitative (and not cardinality) reasons. Let us note that, e.g., under CH there are
Rothberger subspaces of $\mathbb R$ which are not Hurewicz, e.g.,
Luzin sets, see \cite{Sch96}.

Another result proved in \cite{Zdo05} is that every Menger space in
the Miller model is \emph{Scheepers}, where the latter property is
defined in the same way as the Menger one, with the following
additional requirement: The selected sequence $\{\cup\V_n:n\in\w\}$
should be an $\w$-cover of $X$, which by the definition means that  for every
finite $F\subset X$ there exists $n$ such that $F\subset \cup\V_n$.
It is rather straightforward to check that if all finite powers of
$X$ have the Menger property, then $X$ is Scheepers. Thus the Miller
model was a natural place to address the question
 whether the Menger
property is preserved by finite products of metrizable spaces. The
following theorem  is the main result of \cite{Zdo18}.

\begin{theorem} \label{main_mil}
In the Miller model, the product of any two Menger spaces is Menger
provided that it is Lindel\"of. In particular, in this model the
product of any two Menger metrizable spaces is Menger.
\end{theorem}

Of course Theorem~\ref{main_mil} is not true in ZFC as there are
equalities between cardinal characteristics which yield even
Hurewicz sets of reals with non-Menger product, see our discussion
of products of Hurewicz spaces above.
 Surprisingly, there are also
inequalities between cardinal characteristics which imply that the
Menger property is not productive even for sets of reals, see
\cite{SzeTsa17}. Let us also note that in the Laver model (more
generally, under $\hot b=\hot d$)  the Menger property is not
preserved by products of metrizable spaces, see \cite{RepZdo10} and
\cite{SzeTsaZdo??} for more recent results along these lines.

As we can see in \cite{Zdo18}, a big part of the proof of
Theorem~\ref{main} (as well as all the results from \cite{Zdo05}
mentioned above) requires only the inequality $\hot u<\hot g$ which
holds in the Miller model. However, we do not know the answer to the
following

\begin{question} \label{heike}
Is the Menger property preserved by finite products of metrizable
spaces under $\hot u<\hot g$? If yes, can $\hot u<\hot g$ be
weakened to
 the Filter Dichotomy or NCF?
\end{question}

We refer the reader to \cite[\S~9]{Bla10} for corresponding
definitions. Let us note that it is not enough to assume $\hot
u<\hot d$ in Theorem~\ref{main_mil} as witnessed by the model
constructed in \cite{BlaShe89}, see \cite{SzeTsaZdo??}. Natural
places to look for a possible negative answer to
Question~\ref{heike} are models constructed in \cite{BlaShe87,
MilShe09, Mil??}. Better understanding of these models and possible
simplification of the methods used there would  be at least as
important as a  solution of Question~\ref{heike}.

%

What about the products of Hurewicz spaces in the Miller model? The
product of  finitely many Hurewicz subspaces of $2^\w$ is Menger by
Theorem~\ref{main_mil}, and thus the Miller model seemed for a while to be the
best candidate for a model where the Hurewicz property is preserved
by finite products of metrizable spaces. The next theorem proved in
\cite{RepZdo??} refutes this expectation.

\begin{theorem}  \label{main_gamma}
In the Miller model there are two $\gamma$-subspaces $X,Y$ of $2^\w$
such that $X\times Y$ is not Hurewicz. In particular, in this model
the Hurewicz property is not preserved by finite products of
metrizable spaces.
\end{theorem}

$\gamma$-Spaces were introduced  in \cite{GerNag82} where it was
proved that a Tychonoff space
 $X$ is a $\gamma$-space if and only if $C_p(X)$  has the Frechet-Urysohn
property, i.e., for every $f\in C_p(X)$ and $A\subset C_p(X)$ with
$f\in\bar{A}$ there exists $\la f_n:n\in\w\ra\in A^\w$ converging to
$f$. It is well-known that $\gamma$-spaces have the Hurewicz
property in all finite powers, see, e.g., \cite[Th.~3.6 and Fig.
2]{COC2} and references therein.

The proof of Theorem~\ref{main_gamma} is based on the fact that if
$X\subset 2^\w$,  $X\in V $, and $X$ is a $\gamma$-space in $V$,
then $X$ remains a $\gamma$-space in the forcing extension by an
iteration of the Miller forcing with countable supports. Previously
this was only known for Cohen and random forcing, see
\cite{MilTsaZso16} and \cite{Sch10}, respectively.
 Let us note that
Cohen forcing  fails to preserve Hurewicz spaces, see the discussion
in \cite{MilTsaZso16} after Problem 4.1 therein. This motivates the
following
\begin{question}
Does  Miller forcing (resp. countable support iterations thereof) preserve the Hurewicz property of ground model
spaces which do not contain topological copies of $2^\w$?
 What about (Sierpi\'{n}ski) subspaces of $2^\w$, provided that they exist
 in the ground model\footnote{Recall
 that an uncountable $S\subset 2^\w$ is called a \emph{Sierpi\'{n}ski subspace}
  if $|S\cap Z|\leq\w$ for every $Z$ of Lebesgue measure $0$. It is known
  \cite[Theorem~2.10]{COC2}  that Sierpi\'{n}ski
  subspaces are Hurewicz.}?
\end{question}
Let us note that $(2^\w)^V$ is not Hurewicz in the Miller model. This can be shown in the same way as
that $(2^\w)^V$ is not Menger in the Sacks model. The latter fact is proved in \cite{BarDow12} and is attributed there
to A.~Miller.

\medskip


\noindent Unlike the Laver and Miller models, in the
\textbf{\textit{Sacks model}} introduced in \cite{Sac??} we have the
best possible non-preservation by products results: Countable
support iterations of the Sacks forcing preserve $\gamma$-subspaces
of $\mathbb R$ \cite{RepZdo??}, while being non-Menger is obviously
preserved by any forcing which does not add unbounded reals. Thus in
the Sacks model there exist $\gamma$-spaces $X,Y\subset 2^\w$ with
non-Menger product.

The new feature of the Sacks model is that there are rather few
Menger sets of reals in this model by the following result obtained
in \cite{GarMedZdo??}.

\begin{theorem}\label{main_sacks}
There are $\hot c=\w_2$ many Menger subspaces of $2^\w$ in the Sacks
model.
\end{theorem}

Let us note that by the methods developed in \cite{TsaZdo08} there
are at least $\hot c^{\hot d}$ (resp. $\hot c^{\hot b}$) many Menger
(resp. Hurewicz) subspaces of $2^\w$, and hence
Theorem~\ref{main_sacks} fails in any model of $\hot c=\hot d$
including the Laver and Miller one. However, its analogue for
Hurewicz spaces might still be true in the Miller model as $\hot c^{\hot b}=\w_2^{\w_1}=\w_2$ there. This
motivates the following

\begin{question} \label{q03}
How many Hurewicz spaces $X\subset 2^\w$ are there in the Miller
model? How many concentrated subsets of $2^\w$ are there in the
Miller model? In particular, is there a concentrated subset $X$ of
$2^\w$ of size $\w_2$ in the Miller model? If not, must every
Hurewicz subspace of $2^\w$ of size $\w_2$ contain a topological
copy of $2^\w$?
\end{question}

A subset $X\subset 2^\w$ is called  \emph{concentrated}
on a countable $Q\subset 2^\w$ if $|X\setminus U|\leq\w$ for any
open $U\supset Q.$ It is an easy exercise that if $X$ is
concentrated on $Q\subset X$ then $X$ is Rothberger, and hence in
the Miller model concentrated sets are Hurewicz \cite{Zdo05}, which
naturally relates the different parts of Question~\ref{q03}.

\section{Proofs of Theorems~\ref{main} and \ref{main_mad}.}

We shall first introduce several notions crucial for the proof of Theorem~\ref{main}.
The first two items of the following definition are due to A. Medini. The third one is an ad hoc notion suitable to our needs.

\begin{definition} \label{def01}
\begin{itemize}
\item $X\subset 2^\w$ is called \emph{countably controlled}
if for every $Q\in [X]^\w$ there exists a $G_\delta$-subset $R\supset Q$ of $ 2^\w$
such that $R\subset X$;
\item $X\subset 2^\w$ is called \emph{$\sigma$-compactly controlled}
if for every $\sigma$-compact $Q\subset X$ there exists a $G_\delta$-subset $R\supset Q$ of $ 2^\w$
such that $R\subset X$;
\item $X\subset 2^\w$ satisfies  property
$(\dagger)$ if for every function $R$ assigning to each countable
subset $Q$ of $X$ a $G_\delta$-subset $R(Q)\supset Q$ of $ 2^\w$,
there exists a family $\mathsf Q\subset [X]^\w$ of size $|\mathsf
Q|=\w_1$ and a mapping $K:\mathsf Q\to\mathcal P(2^\w)$ assigning to
 every $Q\in\mathsf Q$ a $\sigma$-compact subset $K(Q)$ of $R(Q)$, such
that
 $X\subset\bigcup_{Q\in\mathsf Q}K(Q)$.
\end{itemize}

\end{definition}
Let us note that under CH any $X\subset 2^\w$ satisfies $(\dagger)$.

As suggested by Definition~\ref{def01} in what follows we shall work a lot with $G_\delta$ subsets of
$ 2^\w$ containing certain countable sets, and hence we need to introduce some auxiliary notation.
Given   a countable sequence  $\la Q_n: n\in\w\ra$  of finite non-empty subsets of $2^\w$
and $u\in\zrost$, we set
$$ R(\la Q_n:n\in\w\ra,u):=\{x\in 2^\w:\exists^\infty n\in\w\exists q\in Q_n\:(|x-q|<1/u(n)) \}, $$
 where
$\exists^\infty$ stands for ``exists infinitely many''
and we identify $ 2^\w$ with the standard Cantor subset of $[0,1]$ whenever we use expressions of the form
$|a-b|$.  As usually, for $x,y\in\w^\w$  notation $x\leq^*y$ means $\{n:x(n)>y(n)\}$ is finite.

\begin{observation} \label{obs01}
Let $Q=\bigcup_{n\in\w}Q_n\subset 2^\w$ be a countable union of finite non-empty sets
and $R\supset Q$ be a $G_\delta$-set. Then there exists $u\in\zrost$
such that $ R(\la Q_n:n\in\w\ra,u)\subset R$
\end{observation}
\begin{proof}
Let $R=\bigcap_{n\in\w}O_m$, where each $O_m$ is open. For every $m$ we can find
$u_m\in\zrost$ such that for every $n$ and $x\in 2^\w$, if
there exists $q\in Q_n$ such that $|q-x|<1/u_m(n)$, then $x\in O_m$.
Now any $u\in \zrost$ such that $u_m\leq^*u$ for all $m$, is as required.
\end{proof}

The following lemma is the key part of the proof of
Theorem~\ref{main}. Its proof is reminiscent of that of
\cite[Theorem~3.2]{MilTsa10}.
 We will use the notation from \cite{Lav76} with  only differences being  that
smaller conditions in a forcing poset  are supposed to
 carry more information about the generic filter, and the ground model is denoted by $V$.
We shall work in $V[G_{\w_2}]$, where $G_{\w_2}$ is
$\IP_{\w_2}$-generic and $\IP_{\w_2}$ is the iteration of length
$\w_2$ with countable supports of the Laver forcing, see
\cite{Lav76} for details. As usually, we shall denote $G_{\w_2}\cap
\IP_\alpha$ by $G_\alpha$, where $\alpha\leq \w_2$. Following
\cite{Lav76}, for a condition $p\in\IP_{\w_2}$ we denote by $p(0)\la
0\ra$ the root of the condition $p(0)$, the latter being the initial
coordinate of $p$.

A subset $C$ of $\w_2$ is called an \emph{$\w_1$-club} if it is
unbounded and for every $\alpha\in\w_2$ of cofinality $\w_1$, if
$C\cap\alpha$ is cofinal in $\alpha$ then $\alpha\in C$.

\begin{lemma} \label{laver}
In the Laver model,  if  $X\subset 2^\w$  has  countably-controlled complement $ 2^\w\setminus X$,
then $X$ satisfies $(\dagger)$.
\end{lemma}
\begin{proof}
Let $R$ be such as in the definition of $(\dagger)$.
By a standard closing-off argument
there exists an $\w_1$-club $C\subset \w_2$ such that for every $\alpha\in C$
the following conditions are satisfied:
\begin{itemize}
\item $X\cap V[G_\alpha]\in V[G_\alpha]$;
\item For every $Q\in [X]^\w\cap V[G_\alpha]$ the set $R(Q)$ is coded in $V[G_\alpha]$; and
\item For every $Q\in [ 2^\w]^\w\cap V[G_\alpha] $ disjoint from $X$
there exists a $G_\delta$ set $O\supset Q$ coded in $V[G_\alpha]$ such that
$O\cap X=\emptyset$.
\end{itemize}

Let us fix $\alpha\in C$.
We claim that
$X\subset \bigcup_{Q\in \mathsf Q}W_Q$,
where $\mathsf Q=[X]^\w\cap V[G_\alpha]$ and
$W_Q=\bigcup\{K:K \mbox{ is a compact subset of }R(Q)\mbox{ coded in }V[G_{\alpha+1}]\}.$
This would complete our proof since it is easy to see (and well-known) that for every
$G_\delta$ subset $R$ of $ 2^\w$ and a family $\mathcal W$ of fewer than $\hot b$ many compact
subspaces of $R$, there exists a $\sigma$-compact subspace $K$ of $R$ containing $\bigcup\mathcal W$.
Note that $\hot b=\w_2$ in $V[G_{\w_2}]$ and there are $\w_1$ many compact sets coded
in $V[G_{\alpha+1}]$.

Suppose that, contrary to our claim, there exists $p\in G_{\w_2}$ and a $\IP_{\w_2}$-name $\name{x}$
such that $p$ forces $ 2^\w\setminus \name{X}$ to be countably-controlled and $\name{x}\in\name{X}\setminus \bigcup_{Q\in \mathsf Q}W_Q$.
By \cite[Lemma~11]{Lav76}
there is no loss of generality
in assuming that $\alpha=0$. Applying \cite[Lemma~14]{Lav76}
to a sequence $\la \name{a}_i:i\in\w\ra$ such that $\name{a}_i=\name{x}$ for all $i\in\w$,
we get a condition $p'\leq p$ such that $p'(0)\leq^0 p(0)$, and a finite set $U_s$
of reals for every $s\in p'(0)$ with $p'(0)\la 0\ra\leq s$, such that for each $\varepsilon>0$,
  $s\in p'(0)$ with $p'(0)\la 0\ra\leq s$, and for all but finitely many
immediate successors $t$ of $s$ in $p'(0)$ we have
\begin{equation} \label{eq00}
  p'(0)_t\bigvid p'\uhr[1,\w_2)\forces \exists u\in U_s\: (|\name{x}-u|<\varepsilon).
\end{equation}
Without loss of generality we may assume that $U_s\setminus X\neq\emptyset$ for all
$s\in p'(0),$ $p'(0)\la 0\ra\leq s$.
 Since $ 2^\w\setminus X$ is countably-controlled,
there exists an enumeration $\la s_n:n\in\w\ra$ of $S':=\{s\in
p'(0): p'(0)\la 0\ra\leq s\}$ and $\phi\in\zrost\cap V $ such that
$R':=R(\la U_{s_n}\setminus X:n\in\w\ra,\phi)$ is disjoint from $X$.
Fix a   $\IP_{\w_2}$-name $\name{N}$ for a natural number such that
$p'$ forces that $|\name{x}-u|\geq 1/\phi(n)$ for all
$n\geq\name{N}$ and $u\in U_{s_n}\setminus \name{X}$. Let $p''\leq
p'$ be such that $p''\uhr[1,\w_2)=p'\uhr[1,\w_2)$ and $p''(0)$ is
obtained from $p'(0)$ by throwing away for every $s_n\in S'$ finitely
many of its successors (as well as all their extensions in $p'(0)$)
so that
\begin{equation} \label{eq01}
  p''(0)_{s_n}\bigvid p''\uhr[1,\w_2)\forces \exists u\in U_{s_n}\: \big(|\name{x}-u|<1/(\phi(n)+1)\big)
\end{equation}
for all $n$ such that $s_n\in S'':=\{s\in p''(0): p''(0)\la 0\ra\leq
s\}$. Replacing $p''$ with a yet stronger condition, if necessary,
we may additionally assume that $p''$ decides $\name{N}$ as some
$N\in\w$, and $s_n\in S''$ implies $n\geq N$. Thus we have that
\begin{equation}\label{eq02}
  p''\forces \forall u\in U_{s_n}\setminus X\: \big(|\name{x}-u|\geq 1/\phi(n)\big)
\end{equation}
for all $n$ such that $s_n\in S''$.
Combining Equations~(\ref{eq00}),(\ref{eq01}), and (\ref{eq02}) we get
that for each $\varepsilon>0$,
  $s\in S''$, and for all but finitely many
immediate successors $t$ of $s$ in $p''(0)$ we have
\begin{equation} \label{eq03}
  p''(0)_t\bigvid p''\uhr[1,\w_2)\forces \exists u\in U_s\cap X\: (|\name{x}-u|<\varepsilon).
\end{equation}
Fix an enumeration $\la t_n:n\in\w\ra\in V$ of $S''$ and find
$\psi\in\zrost\cap V$ such that $R(\la U_{t_n}\cap
X:n\in\w\ra,\psi)\subset R(\bigcup_{n\in\w}U_{t_n}\cap X)$.
Replacing $p''$ with a  stronger condition $p^{(3)}$ obtained
by throwing away
 for every $t_n\in S''$ finitely
many of its successors (as well as all their extensions in $p''(0)$),
 we can get
\begin{equation}\label{eq04}
  p^{(3)}_{t_n}\bigvid p^{(3)}\uhr[1,\w_2)\forces \exists u\in U_{t_n}\cap X\: \big(|\name{x}-u|< 1/2\psi(n)\big)
\end{equation}
 for all  $t_n \geq p^{(3)}(0)\la 0\ra$, $t_n\in p^{(3)}(0)$
(the set of such $t_n$ will be denoted by $S^{(3)}$).
Let $l$ be the first Laver generic, i.e., $l=\bigcap G_1\subset\w^{<\w}$.
Equation~(\ref{eq04}) implies that $p^{(3)}$
forces
$$ \name{x}\in K:=\bigcap_{t_n\in l\cap S^{(3)}}\{z\in 2^\w:\exists u\in U_{t_n}\:(|z-u|\leq 1/2\psi(n))\},$$
and $K$ is obviously a compact subset of $R(\la U_{t_n}\cap
X:n\in\w\ra,\psi)\subset R(\bigcup_{n\in\w}U_{t_n}\cap X)$
coded in $V[G_1]$. This contradicts our assumption on $p$ and thus finishes our proof.
\end{proof}

The next lemma has been proved in \cite{RepZdo17}.  We present
its proof for the sake of completeness.

\begin{lemma} \label{covering_g_delta}
Let $Y\subset  2^\w$ be  Hurewicz  and $Q\subset 2^\w$
 countable. Then for every $G_\delta$-subset $O$ of $ (2^\w)^2$ containing
$Q\times Y$ there exists  a $G_\delta$-subset $R\supset Q$ such that
$R\times Y\subset O$.
\end{lemma}
\begin{proof}
 Without loss of generality we shall assume that $O$ is open. Let us write
$Q$ in the form $\{q_n:n\in\w\}$ and set $O_n=\{z\in 2^\w:\la
q_n,z\ra\in O\}\supset Y$. For every $n$ find a cover $\U_n$ of $Y$
consisting of clopen subsets of $ 2^\w$ contained in $O_n$. Let
$\la\U'_k:k\in\w\ra$ be a sequence of open covers of $Y$ such that
each $\U_n$ appears in it infinitely often. Applying the Hurewicz
property of $Y$ we can find a sequence $\la \V_k:k\in\w\ra$ such
that $\V_k\in [\U'_k]^{<\w}$ and $Y\subset \bigcup_{k\in\w} Z_k$,
where $Z_k=\bigcap_{m\geq k}\cup\V_m$. Note that each $Z_k$ is
compact and $Z_k\subset O_n$ for all $n\in\w $ (because there exists
$m\geq k$ such that $\U'_m=\U_n$, and then
$Z_k\subset\cup\V_m\subset O_n$). Thus $Q\times Y\subset Q\times
(\bigcup_{k\in\w}Z_k)\subset O$. Since $Z_k$ is compact, there
exists for every $k$  an open $R_k\supset Q$ such that $R_k\times
Z_k\subset O$. Set $R=\bigcap_{k\in\w}R_k$ and note that $R\supset
Q$ and $R\times Y\subset R\times\bigcup_{k\in\w}Z_k\subset O$.
\end{proof}

The next lemma demonstrates the relation between $(\dagger)$
and products with Hurewicz spaces.

\begin{lemma} \label{main1}
 Suppose that $\hot b>\w_1$. Let $Y \subset  2^\w$ be a Hurewicz  space
and $X\subset  2^\w$  satisfy $(\dagger)$.  Then $X\times Y$ is Hurewicz.
\end{lemma}
\begin{proof}
Fix a sequence $\la \U_n:n\in\w\ra$ of  covers of $X\times Y$
by clopen subsets of $ (2^\w)^2$.
Thus $X\times Y\subset\bigcap_{n\in\w}O_n$, where $O_n=\cup\U_n$.
By   Lemma~\ref{covering_g_delta} for every
 $Q\in [X]^\w $
we can find  a $G_\delta$-subset $R_{Q}\supset Q$  such that
$R_{Q}\times Y\subset \bigcap_{n\in\w}O_n$. Since $X$ satisfies $(\dagger)$, there exists
$\mathsf Q\subset [X]^\w$ of size $|\mathsf Q|=\w_1$,
and for every $Q\in\mathsf Q$ a $\sigma$-compact $K_Q\subset R_Q$,
 such that  $ K =\bigcup\{ K_{Q}  : Q\in\mathsf Q\} $ contains $X$ as a subset.
It follows that $X\times Y\subset K\times Y\subset \bigcap_{n\in\w}O_n$.
It is well-known and easy to see that a product of a compact space and a Hurewicz space
is Hurewicz, as well as that the Hurewicz property is preserved by
unions of fewer than $\hot b$ many subsets of $ 2^\w$, see \cite{Tsa06} and references therein.
Thus $K\times Y$ is Hurewicz and $\U_n$ is an open cover thereof, therefore for all $n\in\w$ there exists $\V_n\in [\U_n]^{<\w}$ such that
$\{\cup\V_n:n\in\w\}$ is a $\gamma$-cover of $K\times Y$, and hence it is
also a $\gamma$-cover of $X\times Y$ as well.
\end{proof}

Finally, we can prove the characterization of Hurewicz subspaces of $ 2^\w$ which holds in the
Laver model and implies Theorem~\ref{main}.

\begin{proposition}\label{prop_char}
In the Laver model, for a subspace $X$ of $ 2^\w$ the following conditions are equivalent:
\begin{enumerate}
\item $X$ is Hurewicz; 
\item $X$ satisfies $(\dagger)$; 
\item $ 2^\w\setminus X$ is countably controlled; 
\item $ 2^\w\setminus X$ is $\sigma$-compactly  controlled; and 
\item $X\times Y$ is Hurewicz for any Hurewicz subspace $Y$ of $ 2^\w$. 
\end{enumerate}
 \end{proposition}
\begin{proof}
The equivalence between $(1)$ and $(4)$ is true in ZFC, see \cite[Theorem
5.7]{COC2} or \cite[Theorem 1.2]{BanZdo08}. The implication $(4)\to(3)$ is obvious outright in ZFC,
while $(3)\to (2)$ is established in Lemma~\ref{laver}. The implication $(2)\to (5)$ is proved in Lemma~\ref{main1}
and thus requires only $\hot b>\w_1$. And finally, $(5)\to (1)$ is again obvious in ZFC, take,
e.g., $Y$ to be a singleton.
\end{proof}

We are now in a position to present the
\medskip

\noindent\textit{Proof of Theorem~\ref{main_mad}.} \ By
Proposition~\ref{prop_char} it is enough to show that $\F(\A)^+$ is
countably controlled: This would give that $\mathcal P(\w)\setminus
\F(A)^+$ is Hurewicz, and the map $\mathcal P(\w)\ni
Z\mapsto\w\setminus Z$ is easily seen to be a homeomorphism between
$\F(A)$ and $\mathcal P(\w)\setminus \F(A)^+$. Let us fix a
countable $\mathcal X\subset\F(\A)^+$ and find countable infinite
$\A_0\subset\A$ such that $A\cap X$ is infinite for all $A\in\A_0$
and $X\in\mathcal X$. It follows that
$$ \{Z\in\mathcal P(\w):\forall A\in\A_0 (|A\cap Z|=\w)\} $$
is a $G_\delta$ subset of $\mathcal P(\w)$ which contains $\mathcal X$ and is contained in $\F(\A)^+$.
\hfill $\Box$
\medskip

By nearly the same argument as at the end of \cite{RepZdo17} ne can prove that  Theorem~\ref{main}
follows from Proposition~\ref{prop_char}. Again, we present
its proof for the sake of completeness.
A family $\F\subset[\w]^\w$ is called a \emph{semifilter} if for every
$F\in\F$ and $X\subset \w$, if $|F\setminus X|<\w$
then $X\in\F$.

The proof of  the first part of Theorem~\ref{main}  uses  the characterization
of the Hurewicz property  obtained in \cite{Zdo05}.
Let  $u=\la U_n : n\in\omega\ra$ be a sequence of subsets of a set $X$.
For every $x\in X$ let  $I_s(x,u,X)=\{n\in\omega:x\in U_n\}$. If every
$I_s(x,u,X)$ is infinite (the collection of all such sequences $u$ will be denoted
by $\Lambda_s(X)$), then we shall denote by $\mathcal U_s(u,X)$
the smallest semifilter on $\omega$ containing all $I_s(x,u,X)$.
By \cite[Theorem~3]{Zdo05},  a Lindel\"of topological space $X$  is Hurewicz if and only if
for every  $u \in\Lambda_s(X)$ consisting of open sets,
  the semifilter $\mathcal U_s(u,X)$ is Hurewicz.
The proof given there also works if we consider only those
$\la U_n : n\in\omega\ra\in\Lambda_s(X)$ such that all $U_n$'s belong to a given base of
$X$.

\medskip

\noindent\textit{Proof of Theorem~\ref{main}.} \
Suppose that $X,Y$ are Hurewicz spaces such that $X\times Y$ is Lindel\"of
and fix  $w=\la U_n\times V_n :n\in\w\ra\in\Lambda_s(X\times Y)$
consisting of open sets.
Set  $u=\la U_n:n\in\w\ra$,  $v=\la V_n:n\in\w\ra$, and note that
$u\in\Lambda_s(X)$ and $v\in\Lambda_s(Y)$.
It is easy to see
that
$$\U_s(w,X\times Y)=\{A\cap B: A\in \U_s(u,X), B\in \U_s(v,Y)\},$$
and hence $\U_s(w,X\times Y)$ is a continuous image of
$\U_s(u,X)\times \U_s(v,Y)$. By \cite[Theorem~3]{Zdo05} both of
latter ones are Hurewicz, considered as subspaces of $2^\w$, and hence their product is a Hurewicz space by
Proposition~\ref{prop_char}. Thus $\U_s(w,X\times
Y)$ is Hurewicz, being a continuous image of a Hurewicz space. It now
suffices to use \cite[Theorem~3]{Zdo05} again. \hfill $\Box$
\medskip


\end{document}